\documentclass[12pt]{amsart}%
\usepackage{lipsum}
\usepackage{amssymb, amsfonts, amsmath, amsthm, amscd}
\usepackage{mathrsfs}
\makeatletter
\g@addto@macro{\endabstract}{\@setabstract}
\newcommand{\authorfootnotes}{\renewcommand\thefootnote{\@fnsymbol\c@footnote}}%
\makeatother
\usepackage{color}

\newtheorem{theorem}{Theorem}[section]
\newtheorem{lemma}[theorem]{Lemma}
\newtheorem{proposition}[theorem]{Proposition}
\newtheorem{corollary}[theorem]{Corollary}
\newtheorem{conjecture}[theorem]{Conjecture}
\theoremstyle{definition}
\newtheorem{definition}[theorem]{Definition}
\newtheorem{example}[theorem]{Example}

\newtheorem{remark}[theorem]{Remark}
\theoremstyle{remark}

\numberwithin{equation}{section}
\renewcommand{\thefootnote}{\fnsymbol{footnote}}
\def\F{\mathcal{F}}
\def\E{\mathcal{E}}

\def\T{\mathcal{T}}

\def\lhd{\vartriangleleft}

\def\Mod{\mathrm{Mod}}
\def\Hom{\mathrm{Hom}}
\def\Iso{\mathrm{Iso}}
\def\Aut{\mathrm{Aut}}
\def\Inn{\mathrm{Inn}}
\def\Out{\mathrm{Out}}
\def\Syl{\mathrm{Syl}}
\def\rank{\mathrm{rank}_l}
\def\Br{\mathrm{Br}}
\def\Bl{\mathrm{Bl}}
\def\Tr{\mathrm{Tr}}
\def\max{\mathrm{max}}

\begin{document}
\title{On the local rank of fusion systems}
\thanks{Supported by NSFC (11971155, 12171142, 11871083)}%    Only \author and \address are required; other information is

%    author one information
\iffalse
\author{Jun Liao}
\address{School of Mathematics and Statistics, 
Hubei University, 
Wuhan, $430062$, P. R. China}
\curraddr{}
\email{jliao@pku.edu.cn}
\thanks{}

\author{Heguo Liu}
\address{School of Mathematics and Statistics, 
Hubei University, 
Wuhan, $430062$, P. R. China}
\curraddr{}
\email{ghliu@hubu.edu.cn}
\thanks{}

\author{Baoshan Wang}
\address{School of Mathematics and System Sciences, 
Beihang University, 
Beijing, $100191$, P. R. China}
\curraddr{}
\email{bwang@buaa.edu.cn}
\thanks{}

%    author two information
\author{Jiping Zhang}
\address{School of Mathematical Sciences, 
Peking University, 
Beijing, $100871$, P. R. China}
\curraddr{}
\email{jzhang@pku.edu.cn}
\thanks{}
\fi

%    \subjclass is required.

\date{}

\dedicatory{}

%    Abstract is required.
\maketitle
\footnotetext{\textit{E-mail addresses}: jliao@pku.edu.cn (J. Liao), ghliu@hubu.edu.cn (H. Liu), bwang@buaa.edu.cn (B. Wang), jzhang@pku.edu.cn (J. Zhang).}
\begin{center}
\normalsize
\textsc{Jun Liao\textsuperscript{1}, Heguo Liu\textsuperscript{1}, Baoshan Wang\textsuperscript{2}\footnote{Corresponding author}, and Jiping Zhang\textsuperscript{3}}
\end{center}
\bigskip
\textsuperscript{1}\textit{School of Mathematics and Statistics, Hubei University, Wuhan $430062$, China}\\
\textsuperscript{2}\textit{School of Mathematics and System Sciences, Beihang University, Beijing $100191$, China}\\
\textsuperscript{3}\textit{School of Mathematical Sciences, Peking University, Beijing $100871$, China}

\begin{abstract}
In this paper we define the notion of local rank for fusion systems
so as to reformulate the Alperin's weight conjecture
in the framework of block fusion systems following the work by Kn\"orr and Robinson.
\end{abstract}

\section{Introduction}
A saturated fusion system $\F$ over a finite
$p$-group $S$  \cite[Definition 1.2]{BLO2} \cite[Proposition I.2.5]{AKO}
is a subcategory of the group category whose objects are the subgroups of $S$, 
and morphisms are group monomorphims satisfying axioms which mimic those satisfied
by morphisms induced by conjugation in some ambient group $G$.
More precisely, conjugation by elements in $S$ need to be in the category, 
every morphism is the composite of an $\F$-isomorphism followed by
an inclusion, 
and furthermore two non-trivial conditions need to be satisfied, 
called the Sylow and extension axiom, 
which we recall below together with some terminology.
We refer to \cite{AKO} and \cite{BLO2} for detailed information, 
and also direct the reader to Puig's original work \cite{Puig:2006a}, 
where he called them Frobenius category.

Let $p$ be a prime and $G$ a finite group with a Sylow $p$-subgroup $S$.
The fusion system $\mathcal{F}_S(G)$ of the finite group $G$ over $S$ is the category whose objects are the subgroups of $S$, 
and whose morphisms are induced by conjugation by elements of $G$, that is
$\mathrm{Hom}_{\mathcal{F}_S(G)}(Q, R)=\mathrm{Hom}_G(Q, R)$ for $Q, R\leq S$.

Let $G$ be a finite group, 
$p$ a prime  dividing the order of $G$, 
and $k$ an algebraically closed field of characteristic $p$.
Let $b$ be a block of $kG$, the indecomposable factor of the group algebra $kG$.
Alperin and Brou\'e introduced the $G$-poset of Brauer pairs associated to a block $b$ of $kG$, 
which then was developed by Brou\'e and Puig.
The elements of this poset are $(kG, b, G)$-Brauer pairs $(Q, e)$ where $Q$ is a $p$-subgroup of  $G$ and
$e$ is a block of $kC_{G}(Q)$ such that $\Br_{Q}(b)e=e$, 
or equivalently $e^{G}=b$.
The inclusion is defined thought the Brauer homomorphism and the first component of any maximal pair is a defect group of $b$.
Fix a maximal $(kG, b, G)$-Brauer pair $(P, e_P )$.
Let $Q\leq P$, there exists a unique block $e_{Q}$ such that $(Q, e_{Q})\leq (P, e_{P})$.
The fusion system $\F$ of $kGb$
(or of $b$) over $P$ is the category $\F_{(P, e_P )}(kG, b, G)$ whose objects are the subgroups of $P$, 
and whose morphisms are induced by conjugation in the $G$-poset of $(kG, b, G)$-Brauer pair contained in $(P, e_{P})$, i.e., 
$\mathrm{Hom}_{\F_{(P, e_P )}(kG, b, G)}(Q, R)=\{c_{g} \mid g\in G, {^{g}(Q, e_{Q})}\leq (R, e_{R})\}$ for $Q, R\leq P$.
Then $\F_{(P, e_P )}(kG, b, G)$ is a saturated fusion system over $P$ whenever $(P, e_P )$ is a maximal $(kG, b, G)$-Brauer pair.

R. Kn\"orr and G.R. Robinson used alternating sum to reformulate the Alperin's weight conjecture in \cite{KR}.
In order to calculate the alternating sum by induction, 
Robinson defined the notion of $p$-local rank of a finite group $G$.
Motivated by this, we seek for a proper notion of normalised centric-radical chain for fusion systems
in order to give a reformulation of Alperin's weight conjecture in the context of
block fusion systems.
We refer to \cite[Section 10.7]{L} for more information on Alperin’s weight conjecture and alternating sums.

The paper is organized as follows.
In Section 2 we collect some definitions and notion for fusion systems and give some notations.
In Sections 3 we define the local rank for fusion systems and present some properties of the related subgroups.
We give a reformulation of Alperin's weight conjecture in Section 4.
Finally, in Section 5 we give necessary and suficient conditions for a fusion system with small local rank.

\section{Preliminaries}

In this section we collect the definitions for fusion systems and some terminology.
\subsection{Fusion systems}

The original notion of an abstract fusion system is due to Puig. The following modified version by Broto-Levi-Oliver \cite[1.1]{BLO2} is  equivalent to Puig's definition which is called divisible category \cite{Puig:2006a}.

\begin{definition}[\cite{BLO2, Puig:2006a}]\label{fus}
A fusion system $\F$ over a finite $p$-group $S$ is a subcategory of the group category, 
where the objects of $\F$ is the set of the subgroups of $S$, and which satisfies the following two properties for all $P, Q\leq S$:

\begin{itemize}
  \item[(1)] $\Hom_S(P, Q)\subseteq\mathrm{Mor}_\F(P, Q)\subseteq \mathrm{Inj}(P, Q)$, where $\mathrm{Inj}(P, Q)$ is the set of the injective group homomorphisms from $P$ to $Q$; and
  \item[(2)] Each $\F$-morphism is the composite of an $\F$-isomorphism followed by an inclusion.
\end{itemize}
 \end{definition}

As usual, we write $\Hom_\F(P, Q)=\mathrm{Mor}_\F(P, Q)$, 
$\mathrm{Iso}_\F(P, Q)=\Hom_\F(P, Q)$ if $|P|=|Q|$.
$\mathrm{Aut}_\F(P)=\Hom_\F(P, P)$, $\mathrm{Out}_\F(P)=\mathrm{Aut}_\F(P)/\Inn(P)$.
We say that two subgroups $P, Q\leq S$ are $\F$-conjugate if they are isomorphic as objects of the category $\F$. Let $P^\F$ denote the set of all subgroups of $S$ which are $\F$-conjugate to $P$.

\subsection{Saturated fusion systems}

Let $\F$ be a fusion system over $S$ and $Q\leq S$.

$\bullet$ $Q$ is fully $\F$-automised if $\mathrm{Aut}_S(Q)\in \mathrm{Syl}_p(\mathrm{Aut}_\F(Q))$.

$\bullet$ $Q$ is $\F$-receptive if it has the following property: for each $P\in Q^\F$ and each $\varphi\in \mathrm{Iso}_\F(P, Q)$, there is $\tilde{\varphi}\in \Hom_\F(N_\varphi, S)$ such that $\tilde{\varphi}|_P= \varphi$, where  $N_\varphi=\{g\in N_S(P)\mid  {^\varphi c_g}\in \mathrm{Aut}_S(Q)\}$.

The following version of the definition of saturated fusion systems is due to Roberts-Shpectorov, see \cite[I.2.2]{AKO}.
This definition is equivalent to that given by Broto-Levi-Oliver \cite[1.2]{BLO2}, 
and Puig's Frobenius category \cite[2.9]{Puig:2006a}.

\begin{definition}\label{sfus}
Let $\F$ be a fusion system over a finite $p$-group $S$. The fusion system $\F $ is said to be saturated if each subgroup of $S$ is $\F$-conjugate to a subgroup which is fully $\F$-automised and $\F$-receptive.
\end{definition}

\subsection{Alperin's fusion theorem for saturated fusion systems }

The Alperin fusion theorem for fusion systems plays a very important role in the study of fusion systems.
Let $\F$ be a fusion system over $S$ and $Q\leq S$.

$\bullet$ $Q$ is fully $\F$-normalised, if $|N_{S}(Q)|\geq |N_{S}(P)|$ for all $P\in Q^{\F}$.

$\bullet$ $Q$ is $\F$-centric, if $C_{S}(P)=Z(P)$ for all $P\in Q^{\F}$.

$\bullet$ $Q$ is $\F$-radical, if $O_{p}(\Out_{\F}(Q))=1$.

$\bullet$ $Q$ is $\F$-essential, if $Q$ is fully $\F$-normalised, $\F$-centric, and $\Out_\F(Q)$ has a strongly $p$-embedded subgroup.

As usual, we write $\F^f$, $\F^c$, $\F^r$  and $\F^e$ the set of fully $\F$-normalised, $\F$-centric, $\F$-radical, and $\F$-essential subgroups, respectively. Write $\F^{fcr}$ the set of
subgroups which is fully $\F$-normalised, $\F$-centric and $\F$-radical.

\begin{theorem}\label{fusthm}
Let $\F$ be a saturated fusion system over a finite $p$-group $S$. Then
$\F=\langle\Aut_\F(P)\mid P \mbox{ is } \F\mbox{-essential or } P=S \rangle.$
\end{theorem}

\section{centric-radical chains}

\begin{lemma}\label{equivalent}
Let $\F$ be a saturated fusion system over a finite $p$-group $S$.
If $Q\leq S$ is fully $\F$-normalised, $\F$-centric and $\F$-radical, 
then $O_p(N_\F(Q))=Q$.
\end{lemma}

\begin{proof}
Assume that $Q$ is fully normalised in $\F$ which  is $\F$-centric and $\F$-radical.
Let $O_p(N_\F(Q))=R$.
Then each $\varphi\in \Aut_\F(Q)\subseteq N_\F(Q)$ can be extended to
$\tilde{\varphi}\in\Aut_\F(R)$.
Thus $\Aut_R(Q)\lhd\Aut_\F(Q)$.
Since $Q$ is $\F$-radical and $\F$-centric, 
we see that $R\leq QC_S(Q)=Q$.
Hence $O_p(N_\F(Q))=Q$.
\end{proof}

By Lemma \ref{equivalent}, a normal subgroup of $S$ which is $\F$-centric and $\F$-radical is an $\mathfrak{m}$-radical subgroup in the sense  of \cite[Definition 3.1]{WZ}.

\begin{definition}
A subgroup $Q\leq S$ is said to be weakly centric-radical in $\F$ if $O_p(N_\F(Q))=Q$.
\end{definition}

%Note that $Q$ is weakly centric-radical in $\F$ if, for some fully normalised $\F$-conjugate $P$ of $Q$, 
%$O_p(N_\F(P))=P$, as $N_\F(P)\cong N_\F(R)$ whenever $R\in Q^{\F}\cap \F^{f}$, which induced by some isomorphism $\chi\in\Hom_{\F}(N_{S}(P), N_{S}(R))$.
%Write \[\F^{\mathfrak{cr}}=\{ Q\leq S\mid \mbox{there exists }P\in Q^{\F}\cap \F^{f}\mbox{ such that }O_p(N_\F(P))=P\}\] for the set of the weakly centric-radical subgroups in $\F$.
%Then we have $Q^{\F}\subseteq \F^{\mathfrak{cr}}$ whenever $Q\in \F^{\mathfrak{cr}}$.

Recall that a subgroup $Q\leq G$ of a group $G$ is $p$-radical in $G$ if $O_p(N_G(Q))=Q$.
In fact, from the viewpoint of category, 
this condition can be represented in the transporter category, 
so that the definition is analogous to the definition of $\F$-radical subgroup.
More precisely, let $G$ be a finite group and $S\in\Syl_p(G)$.
The transporter category $\T=\T_S(G)$ over $S$ is a category whose objects are the subgroups of $S$ with the set $\mathrm{Mor}_\T(P, Q)=\{g\in G\mid{^g}P\leq Q\}$ of morphisms from $P$ to $Q$.
Then $\Aut_\T(Q)=\mathrm{Mor}_\T(Q, Q)=N_G(Q)$ and $\Inn_\T(Q):=\mathrm{Mor}_Q(Q, Q)=Q$.
Thus a $p$-subgroup $Q$ is $p$-radical if $O_p(\Aut_\T(Q))=\Inn_\T(Q)$, 
i.e. $O_p(\Out_\T(Q))=O_p(\Aut_\T(Q)/\Inn_\T(Q))=1$.
This is analogous to the definition of $\F$-radical subgroup satisfying
$O_p(\Aut_\F(Q))=\Inn(Q)$.

Given a chain of subgroups $\sigma\colon Q_0<Q_1< \cdots<Q_n$ of $S$, 
define the length $|\sigma|=n$, 
the final subgroup $V^{\sigma}=Q_{n}$, 
the initial subgroup $V_\sigma=Q_0$, 
the $i$-th initial subchain $\sigma_i\colon Q_0<Q_1< \cdots<Q_i$.
%We always assume that the chain has the initial subgroup $V_\sigma=1$ in this paper.

Let  $G$ be a finite group with a Sylow $p$-subgroup $S$.
The stabiliser $G_\sigma=\mathrm{Stab}_G(\sigma)$ of $\sigma$ in $G$ is defined by 
\[G_\sigma=\bigcap_{i=0}^nN_G(Q_i)=\bigcap_{i=0}^n\Aut_\T(Q_i).\]
Let $\F$ be a saturated fusion system over a finite $p$-group $S$.
Similarly, the stabiliser $\mathrm{Stab}_\F(\sigma)$ of $\sigma$ in $\F$ is by definition $\cap_{i=0}^n\Aut_\F(Q_i)$.
We use the same notation $\mathrm{Stab}_\F(\sigma)$ to denote the subsystem of $\F$
over  $N_S(\sigma)=\bigcap_{i=0}^n N_S(Q_i)$ generated by the $\F$-morphisms 
which restricted to $Q_{i}$ is in $\Aut_\F(Q_i)$ for each $0\leq i\leq n$.
Then $\mathrm{Stab}_\F(\sigma)$ is equal to the normaliser $N_\F(\sigma)$ of $\sigma$ in $\F$ as defined in \cite[Section 4]{WZ} recursively by $N_\F(\sigma_i)=N_{N_\F(\sigma_{i-1})}(Q_i)$.
Note that $N_\F(\sigma)$ defined in \cite[Section 4]{WZ} is a fusion system over $N_S(\sigma)$ such that the morphism set $\Hom_{N_\F(\sigma)}(U, V)$ is
\[\{\varphi \in \Hom_\F(U, V)\mid\exists \tilde{\varphi}\in \Hom_\F(UQ_n, VQ_n), \tilde{\varphi}|_U=\varphi, \tilde{\varphi}(Q_i)=Q_i, 0\leq i\leq n\}\]
for $U, V\leq N_S(\sigma)$.

\begin{definition}
Let $\F$ be a saturated fusion system over a finite $p$-group $S$.
A chain of subgroups $\sigma\colon Q_0<Q_1< \cdots<Q_n$ is said to be a weakly centric-radical chain in $\F$, 
if $Q_i$ is weakly centric-radical in ${N_\F(\sigma_{i-1})}$ for $1\leq i\leq n$, 
$N_{\F}(Q_{0})=\F$ and $O_{p}(\F)<Q_{1}$.
A chain of subgroups $\sigma\colon Q_0<Q_1< \cdots<Q_n$ is said to be a normalised centric-radical chain in $\F$, 
if $Q_i$ is fully $N_\F(\sigma_{i-1})$-normalised, $N_\F(\sigma_{i-1})$-centric and $N_\F(\sigma_{i-1})$-radical for $1\leq i\leq n$, $N_{\F}(Q_{0})=\F$ and $O_{p}(\F)<Q_{1}$.
\end{definition}

Note that a weakly centric-radical chain and a normalised centric-radical chain $\sigma\colon Q_0<Q_1< \cdots<Q_n$ in $\F$ is a normal chain, that is to say $Q_{i}\lhd Q_{n}$ for $0\leq i\leq n$.
The condition on the initial subgroup $Q_{0}$ allow us flexible to choose initial normal subgroup between $1$ and $O_{p}(\F)$ for specific purpose.

\begin{lemma}\label{inclusionparabolic}
Let $\F$ be a fusion system over a finite $p$-group $S$
and $R, T\leq S$.
If $N_\F(R)\leq N_\F(T)$, then $N_{RT}(R)\lhd N_\F(R)$.
\end{lemma}

\begin{proof}
By the definition of normaliser fusion system $N_\F(R)$, for $P, Q\leq N_S(R)$, 
\[\Hom_{N_\F(R)}(P, Q)=\{\varphi \in \Hom_\F(P, Q)\mid\exists \tilde{\varphi}\in \Hom_\F(PR, QR), \tilde{\varphi}|_P=\varphi, \tilde{\varphi}(R)=R\}.\]
If $\varphi \in \Hom_{N_\F(R)}(P, Q)$, then there exists $ \phi\in \Hom_\F(PR, QR)$, such that $ \phi|_P=\varphi$ and  $\phi(R)=R$.
Hence $ \phi \in \Hom_{N_\F(R)}(PR, QR)$.
As $N_\F(R)\leq N_\F(T)$, $ \phi\in \Hom_{N_\F(T)}(PR, QR)$.
So there exists $\tilde{\phi}\in  \Hom_\F(PRT, QRT)$, such that $\tilde{\phi}|_{PR}=\phi$ and $\tilde{\phi}(T)=T$.
Thus for each $\varphi \in \Hom_{N_\F(R)}(P, Q)$, 
then there exists $\tilde{\phi}\in  \Hom_\F(PRT, QRT)$, 
such that $\tilde{\phi}|_{P}=\varphi$, $\tilde{\phi}(R)=R$ and $\tilde{\phi}(T)=T$.
Let $\tilde{\varphi}=\tilde{\phi}|_{PN_{RT}(R)}$.
Then $\tilde{\varphi}\in \Hom_\F(PN_{RT}(R), QN_{RT}(R))$ and $\tilde{\varphi}(R)=R$ hence $\tilde{\varphi}\in \Hom_{N_\F(R)}(PN_{RT}(R), QN_{RT}(R))$.
Note that $\tilde{\varphi}|_P=\varphi$ and $\tilde{\varphi}(N_{RT}(R))=N_{RT}(R)$.
Therefore, $N_{RT}(R)\lhd N_\F(R)$.
\end{proof}

\begin{corollary}
Let $\F$ be a fusion system over a finite $p$-group $S$, 
and let $Q\leq S$.
If $O_p(N_\F(Q))=Q$, then $O_p(\F)\leq Q$.
\end{corollary}

\begin{proof}
Note that $N_\F(Q)\leq \F=N_\F(O_p(\F))$.
Then $N_{O_p(\F)Q}(Q)\lhd N_\F(Q)$ by Lemma \ref{inclusionparabolic}.
Since $O_p(N_\F(Q))=Q$, 
it follows that $N_{O_p(\F)Q}(Q)\leq Q$.
Hence $N_{O_p(\F)Q}(Q)=Q$ and $O_p(\F)\leq Q$.
\end{proof}

Let  $\mathfrak{P}=\{ (Q, N_\F(Q) )\mid Q\leq S\}$.
We define a partial order on $\mathfrak{P}$ by
$(R, N_\F(R) )\leq (Q, N_\F(Q) )$ if and only if $R\leq Q$ and $N_\F(R)\leq N_\F(Q)$.

\begin{lemma}
Let $\F$ be a fusion system over a finite $p$-group $S$ and $R\leq S$.
Then $O_p(N_\F(R))=R$  if and only if $(R, N_\F(R))$ is maximal in $\mathfrak{P}$.
\end{lemma}

\begin{proof}
We first assume that $(R, N_\F(R))$ is maximal in $\mathfrak{P}$.
Let $O_p(N_\F(R))=Q$.
Then $R\leq Q$ and $N_\F(R)=N_{N_\F(R)}(Q)\leq N_\F(Q)$.
Hence $(R, N_\F(R) )\leq (Q, N_\F(Q) )$.
By the maximality of $(R, N_\F(R))$, we have $Q=R$.
Hence $O_p(N_\F(R))=R$.

Next we assume that $O_p(N_\F(R))=R$.
If there exists $(Q, N_\F(Q) )$ such that $(R, N_\F(R) )\leq (Q, N_\F(Q) )$, 
then $N_\F(R) \leq  N_\F(Q)$.
By Lemma  \ref{inclusionparabolic}, $N_{QR}(R)\lhd N_\F(R)$.
So $N_{QR}(R)\leq R$ and $Q=R$.
Hence $(R, N_\F(R))$ is maximal in $\mathfrak{P}$.
\end{proof}

Let $\F$ be a saturated fusion system over a finite $p$-group $S$.
Given a subgroup $Q\leq S$, 
for each fully normalised $\F$-conjugate $P$ of $Q$, 
there exists $\chi\in \Hom_\F(N_S(Q), N_S(P))$ such that $\chi(Q)=P$.
Assume that $\Hom_{^\chi N_\F(Q)}(\chi(U), \chi(V))=\chi\Hom_{N_\F(Q)}(U, V)\chi^{-1}$ for $U, V\leq N_S(Q)$.
Then ${^\chi N_\F(Q)}$ is a full subcategory of $N_\F(P)$ over ${^{\chi}N_S(Q)}$.
Thus ${^\chi(Q, N_\F(Q))}\leq (P, N_\F(P))$.

Write $\F^{f\mathfrak{cr}}=\{ Q\leq S\mid  Q\in\F^f, O_p(N_\F(Q))=Q\}=\F^{f}\cap \F^{\mathfrak{cr}}$.
If $Q\in \F^{f\mathfrak{cr}}$ then $(Q, N_\F(Q))$ is maximal in  $\mathfrak{P}$ up to $\F$-conjugation.
Obviously, $(O_p(\F), \F)$ and $(S, N_\F(S))$ are maximal in $\mathfrak{P}$.
Thus $O_p(\F)$ and $S$ are weakly centric-radical subgroups in $\F$.
So $O_p(\F)$ and $S$ are called trivial weakly centric-radical subgroups in $\F$, 
and an $\F^{f\mathfrak{cr}}$-subgroup $Q$ in $\F$ is called canonical if $O_p(\F)<Q<S$.

We call $N_\F(Q)$ a parabolic subsystem of $\F$ if $Q\in \F^{fcr}$ and $O_{p}(\F)<Q$. 
%and $O_p(N_\F(Q))=Q$.
%Note that $O_p(N_\F(Q))=Q$ is equivalent to $N_\F(O_p(N_\F(Q)))=N_\F(Q)$.
Then there is a bijection between the set $\F^{fcr}\setminus \{O_{p}(\F)\}$ and
the set of parabolic subsystems of $\F$ which maps $Q$ to $N_\F(Q)$.
Furthermore, if $\sigma\colon Q_0<Q_1< \cdots<Q_n$ is a normalised centric-radical chain in $\F$, 
then we have a chain $N_\F(Q_0)\supset N_\F(\sigma_1)\supset\cdots\supset N_\F(\sigma_n)$ such that $N_\F(\sigma_{i})$ is a parabolic subsystem of $N_\F(\sigma_{i-1})$ for $i\geq1$.
The chain $N_\F(Q_0)\supset N_\F(\sigma_1)\supset\cdots\supset N_\F(\sigma_n)$ associated with the normalised centric-radical chain $\sigma\colon Q_0<Q_1< \cdots<Q_n$ is called a parabolic chain of $\F$.

\begin{definition}\label{locrank}
Let $\F$ be a saturated fusion system over a finite $p$-group $S$.
We define the local rank $\rank(\F)$ of $\F$ recursively as follows:
\begin{itemize}
  \item[(1)] If $O_p(\F)=S$ then $\rank(\F)=0$;
  \item[(2)] If $O_p(\F)<S$, then $\rank(\F)=1+\max\{\rank(N_\F(R))\mid  R\in\F^{fcr}, O_p(\F)<R\}$.
\end{itemize}
\end{definition}

Then $\rank(\F)$ is the length of a longest parabolic chain in $\F$.
Equivalently, $\rank(\F)$ is the length of a longest normalised centric-radical chain in $\F$.

Define the weakly local rank $\mathrm{rank}_{w}(\F)$ of $\F$ to be the length of a longest weakly centric-radical chain in $\F$.
This also can be defined recursively as follows:
\begin{itemize}
  \item[(1)] If $O_p(\F)=S$ then $\mathrm{rank}_{w}(\F)=0$;
  \item[(2)] If $O_p(\F)<S$, then $\mathrm{rank}_{w}(\F)=1+\max\{\mathrm{rank}_{w}(N_\F(R))\mid  R\in\F^{\mathfrak{cr}}, \mbox{i.e., } O_p(N_\F(R))=R, O_p(\F)<R\}$.
\end{itemize}

Note that a normalised centric-radical chain in $\F$ is a normalised weakly centric-radical chain in $\F$.
Thus $\rank(\F)\leq \mathrm{rank}_{w}(\F)$.
Assume that $\sigma\colon Q_0<Q_1< \cdots<Q_n$ is one of the longest normalised (weakly) centric-radical chain in $\F$. 
Then we have $O_p(N_{\F}(\sigma_{n}))=N_{S}(\sigma_{n})$.
Since $Q_n=O_p(N_{N_{\F}(\sigma_{n-1})}(Q_{n}))$ as $Q_{n}$ is fully normalised and (weakly) centric-radical in $N_{\F}(\sigma_{n-1})$, 
it follows that 
$Q_n=O_p(N_{N_{\F}(\sigma_{n-1})}(Q_{n}))=O_p(N_{\F}(\sigma_{n}))=N_{S}(\sigma_{n})$.
Note that $N_S(\sigma_{n})=N_{N_{S}(\sigma_{n-1})}(Q_{n})$, so we have $Q_{n}=N_{S}(\sigma_{n-1})=N_S(\sigma_{n})$.

\begin{example}
Let $G=S_{8}$ be the symmetric group of degree 8, let $S\in \Syl_{2}(G)$, and let $\F=\F_{S}(G)$.
Then $\mathrm{rank}(\F)=2$ under the definition by \cite[Definition 3.1]{WZ}, 
$\mathrm{rank}_{w}(\F)=\rank(\F)=3$.
\end{example}

\begin{lemma}\label{local induction}
Let $\F$ be a saturated fusion system over a finite $p$-group $S$ and $R\leq S$.
Assume that
$\sigma\colon Q_0<Q_1< \cdots<Q_n$ is a chain in $N_\F(R)$.
Then $O_p(N_{\F}(\sigma_i))\cap N_S(R)\leq O_p(N_{N_\F(R)}(\sigma_i))$.
\end{lemma}

\begin{proof}
Since $N_{N_\F(R)}(\sigma_i)=N_{N_\F(\sigma_i)}(R)\leq N_\F(\sigma_i)=N_{N_\F(\sigma_i)}(O_p(N_{\F}(\sigma_i)))$, it follows that $O_p(N_{\F} (\sigma_i)) \cap N_{N_S(\sigma_i)}(R)\lhd N_{N_\F(\sigma_i)}(R)$ by Lemma \ref{inclusionparabolic}.
That is $O_p(N_{\F}(\sigma_i))\cap N_{S}(R)\lhd N_{N_\F(R)}(\sigma_i)$.
Hence $O_p(N_{\F}(\sigma_i))\cap N_S(R)\leq O_p(N_{N_\F(R)}(\sigma_i))$.
\end{proof}

Analogous to \cite[Proposition 3.1]{AE}, we can show that weakly local rank of a local subsystem of $\F$ is less than or equal to that of $\F$.

\begin{proposition}
Let $\F$ be a saturated fusion system over a finite $p$-group $S$.
Assume that
$\sigma\colon Q_0<Q_1< \cdots<Q_n$ is a weakly centric-radical chain in $N_\F(R)$, 
where $R$ is fully normalised in $\F$.
Then there exists $\sigma^*\colon P_0<P_1< \cdots<P_n$ which is a
weakly centric-radical chain in $\F$.
Consequently, $\mathrm{rank}_{w}(N_\F(R))\leq \mathrm{rank}_{w}(\F)$ for fully normalised subgroup $R$ in $\F$.
In particular, if $R\unlhd S$
then $\sigma\colon O_{p}(\F)<Q_1< \cdots<Q_n$ is also a weakly centric-radical chain in $\F$.
\end{proposition}

\begin{proof}
Let $\E=N_\F(R)$ and $T=N_S(R)$.
Let $(P_1, N_\F(P_1))$ be maximal in $\mathfrak{P}$ such that  $(Q_1, N_\F(Q_1))\leq (P_1, N_\F(P_1))$.
Then $P_1\in\F^{\mathfrak{cr}}$ such that $N_\E(Q_1)\leq N_\F(Q_1)\leq N_\F(P_1)$.

We claim that $N_\E(Q_1)\leq N_\E(P_1)$.
In fact, let $U, V\leq N_T(Q_1)$, $\varphi\in \Hom_{N_\E(Q_1)}(U, V)$.
There is $\phi_1\in\Hom_\E(UQ_1, VQ_1)$ such that $\phi_1|_{U}=\varphi$
and $\phi_1(Q_1)=Q_1$.
Since $\E=N_\F(R)$, we have $\phi_2\in \Hom_\F(UQ_1R, VQ_1R)$  such that
$\phi_2|_{UQ_1}=\phi_1$ and $\phi_2(R)=R$.
Thus $\phi_2\in \Hom_{N_\F(R)}(UQ_1R, VQ_1R)=\Hom_\E(UQ_1R, VQ_1R)$, 
so $\phi_2\in \Hom_{N_\E(Q_1)}(UQ_1R, VQ_1R)$ as $\phi_2(Q_1)=Q_1$.
Since $N_\E(Q_1)\leq N_\F(P_1)$, 
there is $\tilde{\varphi}\in\Hom_\F(UQ_1RP_1, VQ_1RP_1)$ such that $\tilde{\varphi}|_{UQ_1R}=\phi_2$ and $\tilde{\varphi}(P_1)=P_1$.
We have $\tilde{\varphi}\in\Hom_{N_\E(P_1)}(UQ_1RP_1, VQ_1RP_1)$.
Thus $\varphi=\tilde{\varphi}|_{U}\in\Hom_{N_\E(P_1)}(U, V)$ and $N_\E(Q_1)\leq N_\E(P_1)$ as claimed.

Then by Lemma \ref{inclusionparabolic}, 
$P_1\cap N_T(Q_1)\lhd N_\E(Q_1)$.
Since $Q_1$ is a weakly centric-radical subgroup of $\E$, 
we have $O_p(N_\E(Q_1))=Q_1$.
Thus $P_1\cap N_T(Q_1)\leq Q_1$.
Hence
$Q_1=P_1\cap N_T(Q_1)=P_1\cap T$.
Since $O_{p}(\F)\cap T\leq O_{p}(\E)<Q_1=P_1\cap T$, 
we have $O_{p}(\F)<P_1$.

Assume that we have a weakly centric-radical chain $\sigma^*\colon P_0<P_1< \cdots<P_m$  in $\F$ for $1\leq m<n$, 
such that $N_\E(\sigma_i)\leq N_\F(\sigma^*_i)$  and $Q_i=P_i\cap N_T(\sigma_i)=P_i\cap N_T(\sigma_{i-1}) $ for all $1\leq i\leq m$.
Note that $Q_{m+1}\leq Q_n\leq N_T(\sigma_n)\leq N_T(\sigma_m)\leq N_S(\sigma^*_m)$ and
$(Q_{m+1}, N_{N_\F(\sigma^*_{m})}(Q_{m+1}))\leq (P_{m}Q_{m+1}, N_{N_\F(\sigma^*_{m})}(P_{m}Q_{m+1}))$.
Let $(P_{m+1}, N_{N_\F(\sigma^*_{m})}(P_{m+1}))$ be maximal in $\mathfrak{P}$ of $N_\F(\sigma^*_{m})$ such that  $(Q_{m+1}, N_{N_\F(\sigma^*_{m})}(Q_{m+1}))\leq (P_{m+1}, N_{N_\F(\sigma^*_{m})}(P_{m+1}))$ and $P_{m}Q_{m+1}\leq P_{m+1}$.
Then we have $P_{m+1}\in {N_\F(\sigma^*_m)}^{\mathfrak{cr}}$.
Since $N_\E(\sigma_{m+1})=N_{N_\E(\sigma_{m})}(Q_{m+1})\leq N_{N_\F(\sigma^*_{m})}(Q_{m+1})\leq N_{N_\F(\sigma^*_{m})}(P_{m+1})$, 
we see that 
$N_\E(\sigma_{m+1})\leq N_\F(\sigma^*_{m+1})$ and
$N_{N_\E(\sigma_{m})}(Q_{m+1})\leq N_{N_\E(\sigma_{m})}(P_{m+1})$.
Then by Lemma \ref{inclusionparabolic}, 
$P_{m+1}\cap N_{N_T(\sigma_{m})}(Q_{m+1})\lhd N_{N_\E(\sigma_{m})}(Q_{m+1})$.
Hence $Q_{m+1}=P_{m+1}\cap N_{N_T(\sigma_{m})}(Q_{m+1})=P_{m+1}\cap N_T(\sigma_{m})$.
Since $Q_{m}=P_m\cap N_T(\sigma_{m})<P_{m+1}\cap N_T(\sigma_{m})=Q_{m+1}$, 
we have $P_{m}< P_{m+1}$.
By induction, we prove that there exists a weakly centric-radical chain $\sigma^*\colon P_0<P_1< \cdots<P_n$ in $\F$.
Hence, $\mathrm{rank}_{w}(N_\F(R))\leq \mathrm{rank}_{w}(\F)$ for fully normalised subgroup $R$ in $\F$.

%By lemma \ref{local induction}, $O_p(N_{\F}(\sigma_i))\cap N_S(R)\leq O_p(N_{N_\F(R)}(\sigma_i))$. Since $O_p(N_{N_\F(R)}(\sigma_i))=Q_i$ and $N_S(R)=S$, we have $O_p(N_{\F}(\sigma_i))=Q_i$.
In particular, if $R\unlhd S$ then $T=N_{S}(R)=S$.
So $P_{i}=Q_{i}$ for $i=1, 2\dots, n$.
Thus $\sigma\colon O_{p}(\F)<Q_1< \cdots<Q_n$ is a weakly centric-radical chain in $\F$.
\end{proof}

\begin{proposition}
Assume $\F_{1}$ and $\F_{2}$ are saturated fusion systems over $p$-groups $S_{1}$ and $S_{2}$, respectively.
Then $\mathrm{rank}_{w}(\F_{1}\times \F_{2})=\mathrm{rank}_{w}(\F_{1})+\mathrm{rank}_{w}(\F_{2})$ and
$\rank(\F_{1}\times \F_{2})=\rank(\F_{1})+\rank(\F_{2})$.
\end{proposition}

\begin{proof}
Let $S=S_{1}\times S_{2}$ and $\F=\F_{1}\times\F_{2}$.
For each subgroup $P\leq S$, 
let $P_{i}\leq S_{i}$ be the image of $P$ under projection onto $S_{i}$.
By \cite[2.5]{Asch}, $N_{\F}(P)\leq N_{\F}(P_{1}P_{2})=N_{\F_{1}}(P_{1})\times N_{\F_{2}}(P_{2})$.
If $O_{p}(N_{\F}(P))=P$, then $P=P_{1}\times P_{2}$ by Lemma \ref{inclusionparabolic}.
By \cite[Lemma 5.4]{WZ}, $O_{p}(N_{\F_{i}}(P_{i}))=P_{i}$.
Assume that $P\in\F^{f}$ and $\varphi_{i}\in \Hom_{\F_{i}}(N_{S_{i}}(P_{i}), N_{S_{i}}(Q_{i}))$ such that $Q_{i}\in\F_{i}^{f}$.
Then $\varphi=\varphi_{1}\times \varphi_{2}\in \Iso_{\F}(N_{S}(P), N_{S}(Q))$
as $N_{S}(P)=N_{S_{1}}(P_{1})\times N_{S_{2}}(P_{2})$ and $N_{S}(Q)=N_{S_{1}}(Q_{1})\times N_{S_{2}}(Q_{2})$.
Hence $P_{i}\in \F_{i}^{f}$.
Note that $\Aut_{\F}(P)=\Aut_{\F_{1}}(P_{1})\times \Aut_{\F_{2}}(P_{2})$ and
$C_{S}(P)=C_{S_{1}}(P_{1})\times C_{S_{2}}(P_{2})$.
Then we have $\mathrm{rank}_{w}(\F_{1}\times \F_{2})=\mathrm{rank}_{w}(\F_{1})+\mathrm{rank}_{w}(\F_{2})$ and
$\rank(\F_{1}\times \F_{2})=\rank(\F_{1})+\rank(\F_{2})$.
\end{proof}

\section{A reformulation of Alperin's weight conjecture}

We use the notation as \cite{AKO}, \cite{K} and \cite{KR}.
For convenience, we restate the original notation below.

As usual, when $P$ is a $p$-subgroup of $G$, we let $\Br_P$ denote the Brauer homomorphism $\Br_P\colon (kG)^P\rightarrow kC_G(P)$ which sends an element  $\sum_{x\in G}\alpha_xx$ of $kG$ to the element  $\sum_{x\in C_G(P)}\alpha_xx$ of $kC_G(P)$.

Let $b$ be a block of $kG$.
Fix a maximal $(kG, b, G)$-Brauer pair $(P, e_P )$.
The fusion system $\F$ of $kGb$
(or of block $b$) over $P$ is the category $\F_{(P, e_P )}(kG, b, G)$, see \cite[Section IV.3.4]{AKO} for more details.

For a finite dimensional $k$-algebra $A$, 
let $l(A)$ denote the number of isomorphism classes of simple $A$-modules, 
and $z(A)$ denote the number of isomorphism classes of projective simple $A$-modules.
For a finite group $H$ and a central idempotent $c$ of $kH$, 
$z(kHc)$ is the number of blocks $d$ of $kH$ with trivial defect group contained in $c$, i.e. $cd = d$.

For each $p$-subgroup $Q$ of $G$, denote by $\overline{a}$ the image of an element or a subset $a$ of $kN_G(Q)$ under the canonical surjection
$kN_G(Q)\rightarrow kN_G(Q)/Q$.
Then $\overline{\Br_Q(b)}$ is either 0 or a central idempotent of $kN_G(Q)/Q$.

A pair of the form $(Q, w)$, where $Q$ is a $p$-subgroup of $G$, and $w$ is a block of $kN_G(Q)/Q$ such that $w$ has trivial defect group and $\overline{\Br_Q(b)}w= w$ is called a weight of $kGb$.
Clearly, $G$ acts on the set of $kGb$-weights by conjugation.

\begin{conjecture} [Alperin’s weight Conjecture {\cite{A}}]\label{AlperinConj}
The number of isomorphism classes of simple $kGb$-modules equals the number of $G$-orbits of $kGb$-weights.
That is to say, 
\[ l(kGb)=\sum_{Q\in \mathcal{S}_{p}(G)/G}z(k\overline{N_G(Q)}\overline{\Br_Q(b)}) \]
where $Q$ runs over a set of representatives of the $G$-orbits of the conjugation on the $p$-subgroups $\mathcal{S}_{p}(G)$ of $G$.
\end{conjecture}

By \cite[Proposition 5.5, 5.6 ]{K}, the number of $G$-orbits of $kGb$-weights is equal to
\[ \sum_{Q\in\mathrm{Obj}(\F^{cr})/\F}z(k\overline{N_G(Q)}\overline{\widehat{e_Q}})=\sum_{Q\in\mathrm{Obj}(\F^{cr})/\F}z(k\overline{N_G(Q, e_Q)e_Q})\]
where $\widehat{e_Q}$ is the $N_G(Q)$-orbit sum of $e_Q$, $(Q, e_Q)\leq (P, e_P)$, and 
$Q$ runs over a set of representatives of the $\F$-isomorphism classes of the set $\mathrm{Obj}(\F^{cr})$ of $\F$-centric and $\F$-radical subgroups of $P$.
That is to say, only the $\F$-centric and $\F$-radical subgroups of $P$ has potential contribution to the $kGb$-weights. 

Note that weight conjuctures have been generalised to fusion systems in \cite{KLLS}, while we mainly focus on the fusion systems for blocks of finite group. Thus we still calculate in groups rather than in fusion systems, pay much attention to the alternating sum over chains.

Analogue to \cite[Proposition 3.6]{KR}, 
we can calculate the alternating sum over chains.
Let $\sigma\colon 1<Q_{1}<Q_{2}<\cdots<Q_{n}$ be a fully normalized chain in $\F$,
i.e., $Q_{1}$ is fully normalized in $\F$,
and  $Q_{2}$ is fully normalized in $N_{\F}(Q_{1})$, recursively, $Q_{i+1}$ is fully normalized in $N_{\F}(\sigma_{i})$, where $N_{\F}(\sigma_{i})=N_{N_{\F}(\sigma_{i-1})}(Q_{i})$.
Then we have a chain of $(kG, b, G)$-subpairs
\[\sigma\colon (1, b)<(Q_{1}, e_{1})<(Q_2, e_{2})<\cdots<(Q_n, e_{n}), \]
which is contained in the maximal $(kG, b, G)$-subpair $(P, e_P)$.
Denoted by the stabilizers of the chains respectively by
\[N_{G}(\sigma)=\bigcap_{i=1}^{n}N_{G}(Q_{i}, e_{i}), \quad G_{\sigma}=\bigcap_{i=1}^{n}N_{G}(Q_{i}).\]
Then $N_{G}(\sigma)=\bigcap_{i=1}^{n}N_{G_{\sigma}}(e_{i})=N_{G_{\sigma}}(e_{n})$.
Indeed, for any $g\in N_{G_{\sigma}}(e_{n})$, $^{g}(Q_n, e_{n})=(Q_n, e_{n})$, 
and for each $i$, ${^{g}}(Q_i, e_{i})\leq {^{g}}(Q_n, e_{n})=(Q_n, e_{n})$. Then $^{g}Q_{i}=Q_{i}$ and so $^{g}e_{i}=e_{i}$ by \cite[Theorem IV.2.10]{AKO}.

By \cite[Lemma IV.3.17]{AKO}, $e_{i}$ is a block of $G_{i}=\cap_{j=1}^{i}N_{G}(Q_{j}, e_{j})$.
And $e_{1}$ is a block of $N_G(Q_{1}, e_{1})$.
Also we have a chain of $(kN_G(Q_{1}, e_{1}), e_{1}, N_G(Q_{1}, e_{1}))$-subpairs
\[\sigma^1\colon(1, e_{1})<(Q_2, e_{2})<\cdots<(Q_n, e_{n}).\]
Furthermore, as $Q_{1}$ is fully normalized in $\F$, by \cite[Theorem IV.3.19]{AKO}, we have
\[N_\F(Q_{1})=\F_{(N_P(Q_{1}), e_{N_P(Q_{1})})}(kN_G(Q_{1}, e_{1}), e_{1}, N_G(Q_{1}, e_{1})).\]
In general, let \[\sigma^i\colon(1, e_{i})<(Q_{i+1}, e_{i+1})<\cdots<(Q_n, e_{n}).\]
It is a chain of $(kG_{i}, e_{i}, G_{i})$-subpairs and $e_{i}\in C_G(Q_i)^{G_{i}}$.
%
%Let $C_G(\overline{Q_i}, \overline{e_{Q_{i}}})$ be the full preimage of $C_{\overline{N_G(Q, e_{Q})}}(\overline{Q_i}, \overline{e_{Q_{i}}})$ in $N_G(Q, e_{Q})$.
%Then we have  $C_G(\overline{Q_i}, \overline{e_{Q_{i}}})=\{g\in N_G(Q, e_{Q})\mid [g, Q_i]\leq Q, {^{g}e_{Q_{i}}}=e_{Q_{i}}\}$.
%Hence $e_{Q_{i}}$ is a block of $C_G(\overline{Q_i}, \overline{e_{Q_{i}}})$ by \cite[Lemma IV.3.17]{AKO}.
%So $\overline{e_{Q_i}}$ is a central idempotent of $C_{\overline{N_G(Q, e_{Q})}}(\overline{Q_i}, \overline{e_{Q_{i}}})$.
%
Note that the chain
\[\sigma\colon (1, b)<(Q_{1}, e_{1})<(Q_2, e_{2})<\cdots<(Q_n, e_{n}), \]
associates with a chain
\[\tau\colon (G, b)>(G_{1}, e_1)>(G_{2}, e_{2})>\cdots>(G_{n}, e_{n}), \]
which corresponds to a descending chain of block fusion systems
\[\F\supset N_\F(Q_{1})\supset N_\F(\sigma_2)\supset\cdots\supset N_\F(\sigma_n).\]
Here $e_{i}$ is a block of $G_{i}$. Recall that $Q_{i}$ is fully normalized in $N_{\F}(\sigma_{i-1})$,
then by \cite[Theorem IV.3.19]{AKO} we have 
\[N_\F(\sigma_i)=\F_{(N_{P}(\sigma_{i}), e_{N_{P}(\sigma_{i})})}(kG_{i}, e_{i}, G_{i}).\]
%
%
%Let $\sigma$ be a chain of $b$-Brauer pairs which contained in $(P, e_{P})$, i.e.,
%\[\sigma\colon (1, b)<(Q_{1}, e_{1})<(Q_2, e_{2})<\cdots<(Q_n, e_{n}), \]
%We still denoted by $\sigma$ the chain of subgroups of the first component of $\sigma$.
%\[\sigma\colon 1<Q_{1}<Q_{2}<\cdots<Q_{n}, \]
%Set $Q=Q_{1}$ and $R=Q_{n}$ for easy notation.
%
%Let $G_{\sigma_{i}}=\cap_{j=1}^iN_G(Q_j, e_{Q_{j}})$ be the stabiliser of subchain $\sigma_{i}$ in $G$.
%Let $b_{\sigma_{i}}=kG_{\sigma_{i}}e_{Q_{i}}$.
%Write $G_{\sigma}=G_{\sigma_{n}}$ and $b_{\sigma}=b_{\sigma_{n}}$.
%
%Let $e_{n_{0}}=e_{n}$, i.e., $(Q_{n}, e_{n_{0}})\leq (P, e_{P})$. 
%

Note that $G_{\sigma}\leq N_{G}(Q_{n})$ acts on the blocks of $C_{G}(Q_{n})$ associated with $b$, 
which is the set of  $e_{n_{i}}\in \Bl(C_{G}(Q_{n}))$ satisfying $\Br_{Q_{n}}(b)e_{n_{i}}=e_{n_{i}}$.
Therefore,
\[\Br_{Q_{n}}(b)=\sum_{i\in I} \widehat{e_{n_{i}}},\] 
where $e_{n_{i}}, i\in I$ 
runs over a representatives of $G_{\sigma}$-orbit of $C_{G}(Q_{n})$ blocks in $\Br_{Q_{n}}(b)$, where
\[\widehat{e_{n_{i}}}=\Tr_{N_{G_{\sigma}}(e_{n_{i}})}^{G_{\sigma}}(e_{n_{i}}).\]

\begin{lemma}
Let $e_{n_{i}}, i\in I$ be a representatives of $G_{\sigma}$-orbits of $C_{G}(Q_{n})$ blocks in $\Br_{Q_{n}}(b)$.
Let $\sigma(i)$ be the chain 
\[\sigma(i)\colon (1, b)<(Q_{1}, e_{1i})<(Q_2, e_{2i})<\cdots<(Q_n, e_{ni}).\]
Let $ (P_{i}, e_{P_{i}})$ be a maximal  $b$-Brauer pairs containing $(Q_{n}, e_{n_{i}})$.
Then there exist $g_{i}\in G$ such that $ ^{g_{i}}(P_{i}, e_{P_{i}})=(P, e_{P})$,
and  $^{g_{i}}\sigma(i), i\in I$ is a representatives of $\F$-conjugacy classes of 
\[^{g}\sigma\colon (1, b)<(^{g}Q_{1}, e_{^{g}Q_{1}})<(^{g}Q_2, e_{^{g}Q_2})<\cdots<(^{g}Q_n, e_{^{g}Q_n})\subseteq (P, e_P),\]
whenever $^{g}Q_n\subseteq  P$.
\end{lemma}

\begin{proof}
Let $ (P_{i}, e_{P_{i}})$ be a maximal  $b$-Brauer pair contains $(Q_{n}, e_{n_{i}})$, that is to say, $(Q_{n}, e_{n_{i}})\leq  (P_{i}, e_{P_{i}})$.
Then by \cite[Theorem IV.2.20]{AKO} there is $g_{i}\in G $, such that $ ^{g_{i}}(P_{i}, e_{P_{i}})=(P, e_{P})$.
Let $\sigma(i)$ be the chain 
\[\sigma(i)\colon (1, b)<(Q_{1}, e_{1i})<(Q_2, e_{2i})<\cdots<(Q_n, e_{ni}).\]
Then $^{g_{i}}\sigma(i)$ is a chain in $\F$.

First, we show  that $^{g_{i}}\sigma(i)$ and  $^{g_{j}}\sigma(j)$ are not $\F$-conjugate whenever $i\neq j$.
In fact, if there exists  $g\in G$ such that $^{gg_{i}}\sigma(i)= {^{g_{j}}}\sigma(j)$, then
$^{g_{j}^{-1}gg_{i}}\sigma(i)=\sigma(j)$. So $g_{j}^{-1}gg_{i}\in G_{\sigma}$ and $^{g_{j}^{-1}gg_{i}}e_{ni}=e_{nj}$, which is a contradiction as $e_{ni}$ is representatives of $G_{\sigma}$-orbits.

We claim that $^{g_{i}}(Q_{n}, e_{n_{i}}), i\in I$ is a representatives of $\F$-conjugacy classes whose first term is $G$-conjugate to $Q_{n}$. 
Indeed, for any $g\in G$ and $(1, b)\subseteq (^{g}Q_{n}, e_{^{g}Q_{n}})\leq (P, e_{P})$, we have 
$(1, b)\subseteq (Q_{n}, {^{g^{-1}}}\!e_{^{g}Q_{n}})$ and $\Br_{Q_{n}}(b)\, {^{g^{-1}}}\!e_{^{g}Q_{n}}={^{g^{-1}}}\!e_{^{g}Q_{n}}$. 
So there is $x\in G_{\sigma}$ satisfying $^{xg^{-1}}\!e_{^{g}Q_{n}}=e_{ni}$ for some $i\in I$.
Thus $^{xg^{-1}}(^{g}Q_{n}, e_{^{g}Q_{n}})=(Q_{n}, e_{ni})$.
Since $^{g_{i}}(Q_{n}, e_{ni})\leq (P, e_{P})$, it follows that $^{g_{i}xg^{-1}}(^{g}Q_{n}, e_{^{g}Q_{n}})\leq (P, e_{P})$. 
Therefore $^{g_{i}xg^{-1}}(^{g}Q_{n}, e_{^{g}Q_{n}})={^{g_{i}}}(Q_{n}, e_{ni})$, 
so $(^{g}Q_{n}, e_{^{g}Q_{n}})$ and ${^{g_{i}}}(Q_{n}, e_{ni})$ are conjugate in $\F$ as claimed.

Moreover, 
for any $g\in G$ satisfying $^{g}Q_n\subseteq P$, we have a chain in $\F$, 
\[^{g}\sigma\colon (1, b)<(^{g}Q_{1}, e_{^{g}Q_{1}})<(^{g}Q_2, e_{^{g}Q_2})<\cdots<(^{g}Q_n, e_{^{g}Q_n})\subseteq (P, e_P),\]
which conjugated by $g^{-1}$ we have a chain of Brauer pairs, 
\[
(1, b)<(Q_{1}, {^{g^{-1}}e}_{^{g}Q_{1}})\leq (Q_{2}, {^{g^{-1}}e}_{^{g}Q_{2}})\leq\cdots\leq(Q_{n}, {^{g^{-1}}e}_{^{g}Q_{n}}).
\]
As $\Br_{Q_{n}}(b)\,{^{g^{-1}}\!e_{^{g}Q_{n}}}={^{g^{-1}}e}_{^{g}Q_{n}}$, 
we may assume that ${^{xg^{-1}}e}_{^{g}Q_{n}}=e_{ni}$ for some $x\in G_{\sigma}$ and $i\in I$.
Then conjugated by $x\in G_{\sigma}$, we get the chain $\sigma(i)$, 
\[\sigma(i)\colon (1, b)<(Q_{1}, e_{1i})<(Q_2, e_{2i})<\cdots<(Q_n, e_{ni}), \]
Recall that $^{g_{i}}(Q_n, e_{ni})\leq (P, e_{P})$, 
then conjugated by $g_{i}$ we obtain a chain in $\F$, 
\[^{g_{i}}\sigma(i)\colon (1, b)<{^{g_{i}}}(Q_{1}, e_{1i})<{^{g_{i}}}(Q_2, e_{2i})<\cdots<{^{g_{i}}}(Q_n, e_{ni})\subseteq (P, e_P). \]
So $^{g}\sigma$ and $^{g_{i}}\sigma(i)$ are conjugate in $\F$ by $g_{i}xg^{-1}$. 
Therefore  $^{g_{i}}\sigma(i), i\in I$ is a representatives of $\F$-conjugacy classes of $\{^{g}\sigma\mid g\in G\}$ the chain of $G$-orbit of $\sigma$ constituted by subgroups of $P$.
\end{proof}

\begin{lemma}\label{group2fusionConj}
\[
\sum_{\sigma/G}(-1)^{|\sigma|}l(kG_{\sigma}\Br_{Q_n}(b))
=\sum_{\sigma/\F}(-1)^{|\sigma|}l(kN_{G}(\sigma)e_{n}),
\]
where the left hand side runs over the representative of $G$-conjugacy classes of chains,
the right hand side runs over the representative of $\F$-conjugacy classes of chains in $\F$, and 
$\sigma\colon 1<Q_{1}<Q_{2}<\cdots<Q_{n}$ is a chain and $(Q_{n},e_{n})\leq (P, e_{P})$.
\end{lemma}

\begin{proof}

Note that $\Br_{Q_{n}}(b)=\sum_{i\in I} \widehat{e_{ni}}$, we have
\[
l(kG_{\sigma}\Br_{Q_{n}}(b))
=l(kG_{\sigma}\sum_{i\in I} \widehat{e_{ni}})
=\sum_{i\in I} l(kG_{\sigma}\widehat{e_{ni}}).
\]
Since \[^{g_i}\widehat{e_{ni}}={^{g_i}}\Tr_{N_{G_{\sigma}}(e_{n_{i}})}^{G_{\sigma}}(e_{n_{i}})
=\Tr_{N_{G_{^{g_i}\sigma}}(^{g_i}e_{n_{i}})}^{G_{^{g_i}\sigma}}(^{g_i}e_{n_{i}})=\widehat{^{g_i}e_{n_{i}}},\]
we obtain that 
$^{g_{i}}(kG_{\sigma}\widehat{e_{ni}})
=kG_{{^{g_{i}}}\sigma}\widehat{{^{g_{i}}}e_{ni}}$.
Hence 
\[
l(kG_{\sigma}\Br_{Q_{n}}(b))
=\sum_{i\in I} l(kG_{\sigma}\widehat{e_{ni}})
=\sum_{i\in I} l(^{g_i}kG_{\sigma}\widehat{e_{ni}})
=\sum_{i\in I} l(kG_{{^{g_{i}}}\sigma}\widehat{{^{g_{i}}}e_{ni}}).
\]
where the sum runs over the representative of $\F$-conjugacy classes of $G$-conjugates of $\sigma$,
i.e.,
\[
l(kG_{\sigma}\Br_{Q_n}(b))=\sum_{{^{g_i}\!\sigma}\in {^{G}\!\sigma}/\F}l(kG_{^{g_i}\sigma}\widehat{^{g_i}e_{n_{i}}}).
\]
Thus
\[
\sum_{\sigma/G}(-1)^{|\sigma|}l(kG_{\sigma}\Br_{Q_n}(b))=\sum_{\sigma/\F}(-1)^{|\sigma|}l(kG_{\sigma}\widehat{e_n}).
\]
where 
\[\sigma\colon (1, b)<(Q_{1}, e_{1})<(Q_2, e_{2})<\cdots<(Q_n, e_{n})\subseteq (P, e_P),\]
the left hand side runs over a representative of $G$-conjugacy classes of chain of subgroups, 
and the right hand side runs over a representative of $\F$-conjugacy classes of chain of $b$-Brauer pairs which is contained in $(P, e_{P})$.

Since $C_{G}(Q_{n})\leq N_{G_{\sigma}}(e_{n})\leq G_{\sigma}\leq N_{G}(Q_{n})$, it follows from \cite[Lemma IV.3.17]{AKO} that $e_{n}\in \Bl(N_{G_{\sigma}}(e_{n}))$ and $\widehat{e_{n}}\in \Bl(G_{\sigma})$.
Then by \cite[Theorem 6.8.3]{L} or the argument between \cite[Definition 4.5]{K} and \cite[Proposition 4.6]{K}, we have a Morita equivalence $\Mod(kG_{\sigma}\widehat{e_n})\approx\Mod(kN_{G_{\sigma}}(e_{n})e_{n})$. So 
\[l(kG_{\sigma}\widehat{e_n})
=l(kN_{G_{\sigma}}(e_{n})e_{n})=l(kN_{G}(\sigma)e_{n}).\]
Therefore,
\[
\sum_{\sigma/G}(-1)^{|\sigma|}l(kG_{\sigma}\Br_{Q_n}(b))
=\sum_{\sigma/\F}(-1)^{|\sigma|}l(kN_{G}(\sigma)e_{n}).
\]

\end{proof}

\begin{proposition}\label{altsum}
Let $b$ be a block of $kG$.
Then we have
\[
\sum_{\sigma/\F}(-1)^{|\sigma|}l(kN_{G}(\sigma)e_{n})
=l(kGb)-\sum_{1\neq Q/G}\sum_{j=1}^{r}\sum_{\tau/\E_{j}}(-1)^{|\tau|}l(kN_{\overline{N_{G}(Q)}}({\tau})f_{n}),
\]
where $\sigma/\F$ means $\sigma$ run over a representative of $\F$-conjugacy classes of chains in $\F$.
\end{proposition}

\begin{proof}

By \cite[Proposition 3.6]{KR}, 
\[
\sum_{\sigma/G}(-1)^{|\sigma|}l(kG_{\sigma}\Br_{Q_n}(b))
=l(kGb)-\sum_{1\neq Q/G}\sum_{\bar{\sigma}/\overline{N_{G}(Q)}}(-1)^{|\bar{\sigma}|}l(k\overline{N_{G}(Q)}_{\bar{\sigma}}\Br_{\overline{Q_n}}(\overline{\Br_{Q}(b)})).
\]
We may assume that $\overline{\Br_{Q}(b)}=\sum_{j=1}^{r}{c_j}, \,{c_j}\in\Bl(\overline{N_{G}(Q)})$
is the blocks decomposition of the central idempotent $\overline{\Br_{Q}(b)}$.
Thus we have
\[\sum_{\bar{\sigma}/\overline{N_{G}(Q)}}(-1)^{|\bar{\sigma}|}l(k\overline{N_{G}(Q)}_{\bar{\sigma}}\Br_{\overline{Q_n}}(\overline{\Br_{Q}(b)}))
=\sum_{j=1}^{r}\sum_{\bar{\sigma}/\overline{N_{G}(Q)}}(-1)^{|\bar{\sigma}|}l(k\overline{N_{G}(Q)}_{\bar{\sigma}}\Br_{\overline{Q_n}}(c_{j}).\]
Then
\[
\sum_{\sigma/G}(-1)^{|\sigma|}l(kG_{\sigma}\Br_{Q_n}(b))
=l(kGb)-\sum_{1\neq Q/G}\sum_{j=1}^{r}\sum_{\bar{\sigma}/\overline{N_{G}(Q)}}(-1)^{|\bar{\sigma}|}l(k\overline{N_{G}(Q)}_{\bar{\sigma}}\Br_{\overline{Q_n}}(c_{j})).
\]

Let $({S_j}, f_{S_{j}})$ be a maximal $(k\overline{N_{G}(Q)}, c_j, \overline{N_{G}(Q)})$-Brauer pair.
Let  $\E_{j}$ be the fusion system of the block $k\overline{N_{G}(Q)}c_j$.
That is,
\[\E_{j}=\F_{({S_j}, f_{S_{j}})}(k\overline{N_{G}(Q)}, c_j, \overline{N_{G}(Q)}).\]
Then by Lemma \ref{group2fusionConj} we have 
\[\sum_{\bar{\sigma}/\overline{N_{G}(Q)}}(-1)^{|\bar{\sigma}|}l(k\overline{N_{G}(Q)}_{\bar{\sigma}}\Br_{\overline{Q_n}}(c_j))
=\sum_{\tau/\E_{j}}(-1)^{|\tau|}l(kN_{\overline{N_{G}(Q)}}({\tau})f_{n}),\]
where \[\tau\colon (\bar{1}, c_{j})<(\bar{Q}_{2}, f_{2})<(\bar{Q}_{3}, f_{3})<\cdots<(\bar{Q}_{n}, f_{n})\subseteq (S_{j}, {f}_{S_{j}}).\]
Therefore,
\[
\sum_{\sigma/\F}(-1)^{|\sigma|}l(kN_{G}(\sigma)e_{n})
=l(kGb)-\sum_{1\neq Q/G}\sum_{j=1}^{r}\sum_{\tau/\E_{j}}(-1)^{|\tau|}l(kN_{\overline{N_{G}(Q)}}({\tau})f_{n}).
\]

\end{proof}

Note that by \cite[Lemma 3.7]{KR}, $\sum_{\sigma/\F}(-1)^{|\sigma|}l(kN_{G}(\sigma)e_{n})=1$ if the defect group $\delta(b)=1$.
%Without loss of generality, we can assume that $O_p(\F)=1$.
Now we can present a reformulation of Alperin's weight conjecture via block fusion systems.

 \begin{theorem}\label{equi}
 The following two statements are equivalent.
 
\begin{enumerate}
  \item Whenever $G$ is a finite group and $b$ is a block of $kG$, we have
 \[ l(kGb)=\sum_{\F^{cr}/\F}z(k\overline{N_G(Q, e_Q)e_Q}). \]
  \item Whenever $G$ is a finite group and $b$ is a block of $kG$  of positive defect, 
 then
 \[\sum_{\sigma/\F}(-1)^{|\sigma|}l(kN_{G}(\sigma)e_{n})=0.\]
Here $\F=\F_{(P, e_P )}(kG, b, G)$ is the fusion system of block $b$, 
 $(P, e_{P})$ is a maximal $(kG, b, G)$-Brauer pair, 
$(Q, e_{Q})\leq (P, e_{P})$, 
and $\sigma$ runs over a set of representatives for the conjugacy classes of chains in $\F$.
\end{enumerate}
\end{theorem}

\begin{proof}
This is followed by Proposition \ref{altsum} and \cite[Proposition 5.5]{K}.
\end{proof}

\begin{theorem}\label{radchain}
Let $\F=\F_{(P, e_P )}(kG, b, G)$ be the fusion system of block $b$ of  $kG$ with defect group $P$. Let $\F$ be a minimal counterexample of (2) of Theorem \ref{equi} with respect to the order of defect group.
Then  
 \[\sum_{\sigma/\F}(-1)^{|\sigma|}l(kN_{G}(\sigma)e_{n})=\sum_{\sigma\in\F^{cr}/\F}(-1)^{|\sigma|}l(k_{\alpha}\Out_{N_{\F}(\sigma)}(Q_{n})),\]
where $\sigma\in\F^{cr}/\F$ means that $\sigma$ runs over a set of representatives for the $\F$-conjugacy classes of centric-radical chains in $\F$, where $\alpha$ is the K\"ulshammer-Puig classes.
\end{theorem}

\begin{proof}
Assume $Q_{1}$ is not $\F$-centric, or $Q_{1}$ is not $\F$-radical. 
Let $\overline{\widehat{e_{Q_{1}}}}=\sum_{j=1}^{r}{c_j}$ be the block decomposition of the image of the $N_{G}(Q_{1})$-orbit sum of $e_{Q_{1}}$ in $N_{G}(Q_{1})/Q_{1}$. Then each blocks $c_{j}$, which is dominated by $\widehat{e_{Q_{1}}}$, is positive defect by \cite[Proposition 5.6]{K}.
The defect group of block $c_j$ is $\overline{N_{G}(Q_{1})}$-conjugate to a subgroup of $\overline{N_{P}(Q_{1})}$ by \cite[Chapter 5, Theorem 8.7]{NT}.
So we have $|\overline{N_{P}(Q_{1})}|<|P|$. Then by the minimality of hypothesis, 
\[\sum_{\tau/\E_{j}}(-1)^{|\tau|}l(kN_{\overline{N_{G}(Q_{1})}}({\tau})f_{n})=0,\]
for all $1\leq j\leq r$.
Hence all chains beginning with $1<Q_{1}$ have no contribution to the sum.
Therefore, only the chain $\sigma\colon 1<Q_{1}<Q_{2}<\cdots<Q_{n}$ such that $Q_{1}\in \F^{cr}$ has potential contribution to the sum.
By induction, all chains beginning with $1<Q_{1}<Q_{2}<\cdots<Q_{i}$ have potential contribution to the sum only if 
$z(kN_{G_{i-1}}(Q_{i},e_{i})e_{i}/Q_{i})\neq 0$.
Then we have $Q_{i}\in N_{\F}(\sigma_{i-1})^{cr}$ by \cite[Proposition 5.6]{K}.
Hence, only the centric-radical chains $\sigma\colon 1<Q_{1}<Q_{2}<\cdots<Q_{n}$, i.e., $Q_{i}\in N_{\F}(\sigma_{i-1})^{cr}$ for $1\leq i\leq n$, have potential contribution to the sum. Thus we have
\[\sum_{\sigma/\F}(-1)^{|\sigma|}l(kN_{G}(\sigma)e_{n})=\sum_{\sigma\in\F^{cr}/\F}(-1)^{|\sigma|}l(kN_{G}(\sigma)e_{n}).\]

Since we have $Q_{n}\in \F^{c}$, it follows that $\overline{e_{n}}$ is a block of $C_{G}(Q_{n})/Z(Q_{n})$ with defect group 1. Note that $C_{G}(Q_{n})/Z(Q_{n})$ is normal in $N_{G}(\sigma)/Q_{n}$ and $\overline{e_{n}}$ is $N_{G}(\sigma)/Q_{n}$-stable. Then by \cite[Theorem 4.8]{K},
$kN_{G}(\sigma)/Q_{n}\overline{e_{n}}$ is Morita equivalent to $k_{\alpha}\Out_{N_{\F}(\sigma)}(Q_{n})$, where $\alpha$ is the K\"ulshammer-Puig classes, see \cite[Section IV. Section 5.5]{AKO} or \cite[Section 4.4]{K}. So we have $l(k\overline{N_{G}(\sigma)e_{n}})=l(k_{\alpha}\Out_{N_{\F}(\sigma)}(Q_{n}))$ and then $l(kN_{G}(\sigma)e_{n})=l(k_{\alpha}\Out_{N_{\F}(\sigma)}(Q_{n}))$.
Therefore,
 \[\sum_{\sigma/\F}(-1)^{|\sigma|}l(kN_{G}(\sigma)e_{n})=\sum_{\sigma\in\F^{cr}/\F}(-1)^{|\sigma|}l(k_{\alpha}\Out_{N_{\F}(\sigma)}(Q_{n})).\]
\end{proof}

By Theorem \ref{radchain}, Theorem \ref{equi} can be stated in the framework of block fusion systems.

\begin{theorem}\label{equi2}
The following two statements are equivalent.
 
\begin{enumerate}
  \item Whenever $G$ is a finite group and $b$ is a block of $kG$, we have
 \[ l(kGb)=\sum_{\F^{cr}/\F}z(k_{\alpha}\Out_{\F}(Q)). \]
  \item Whenever $\F$ is a fusion system of block of positive defect, 
 then
 \[\sum_{\sigma\in\F^{cr}/\F}(-1)^{|\sigma|}l(k_{\alpha}\Out_{N_{\F}(\sigma)}(Q_{n}))=0,\]
where $\sigma\in\F^{cr}/\F$ means that $\sigma$ runs over a set of representatives for the $\F$-conjugacy classes of centric-radical chains in $\F$, where $\alpha$ is the K\"ulshammer-Puig classes.

\end{enumerate}
\end{theorem}

\begin{proof}
We claim that (2) of Theorem \ref{equi} is equivalent to (2) of Theorem \ref{equi2}.
The proof proceed by induction on the order of defect group.
When $P$ is of order $p$, it is obviously that  
\[\sum_{\sigma/\F}(-1)^{|\sigma|}l(kN_{G}(\sigma)e_{n})=\sum_{\sigma\in\F^{cr}/\F}(-1)^{|\sigma|}l(k_{\alpha}\Out_{N_{\F}(\sigma)}(Q_{n})).\] 
Suppose that assertions are both known to hold for defect groups of order less than $|P|$.
By Theorem \ref{radchain},
\[\sum_{\sigma/\F}(-1)^{|\sigma|}l(kN_{G}(\sigma)e_{n})=\sum_{\sigma\in\F^{cr}/\F}(-1)^{|\sigma|}l(k_{\alpha}\Out_{N_{\F}(\sigma)}(Q_{n})).\] 
Hence we see that assertion (2) of Theorem \ref{equi} holds if and only if  assertion (2) of Theorem \ref{equi2} holds.
 
 \end{proof}

\begin{remark}
Theorem \ref{equi2} shows that the local rank defined in Definition \ref{locrank} should have the honor of being called the local rank of fusion system,
which is useful to reformulate Alperin's weight conjecture via alternating sums in the context of fusion systems. 
\end{remark}

\begin{remark}
The assertion that \[\sum_{\sigma\in\F^{cr}/\F}(-1)^{|\sigma|}l(k_{\alpha}\Out_{N_{\F}(\sigma)}(Q_{n}))=0\] for any saturated fusion system $\F$ over nontrivial finite $p$-group $P$ could be viewed as  Alperin's weight conjecture for fusion systems.
\end{remark}

\begin{remark}
Let $\F$ be a minimal counterexample of (2) of Theorem \ref{equi} with respect to the order of defect group. Then $O_{p}(\F)=1$.
\end{remark}
\begin{proof}
Suppose that $Q=O_{p}(\F)>1$. Note that $N_{\F}(Q)=\F_{(P,e_{P})}(kN_{G}(Q)\widehat{e_{Q}}, N_{G}(Q))$, and $N_{\F}(Q)=\F$.
By Theorem \ref{radchain}, we have
 \[\sum_{\sigma/\F}(-1)^{|\sigma|}l(kN_{G}(\sigma)e_{n})=
 \sum_{\sigma\in\F^{cr}/\F}(-1)^{|\sigma|}l(k_{\alpha}\Out_{N_{\F}(\sigma)}(Q_{n})),\]
and
 \[\sum_{\sigma/\F}(-1)^{|\sigma|}l(kN_{N_{G}(Q)}(\sigma)e_{n})=\sum_{\sigma\in\F^{cr}/\F}(-1)^{|\sigma|}l(k_{\alpha}\Out_{N_{\F}(\sigma)}(Q_{n})).\]
By Lemma \ref{group2fusionConj}, we have 
\[
\sum_{\sigma/\F}(-1)^{|\sigma|}l(kN_{N_{G}(Q)}(\sigma)e_{n})
=\sum_{\sigma/N_{G}(Q)}(-1)^{|\sigma|}l(kN_{G}(Q)_{\sigma}\Br_{Q_n}(\widehat{e_{Q}})).
\]
Hence,
\[
\sum_{\sigma/\F}(-1)^{|\sigma|}l(kN_{G}(\sigma)e_{n})
=\sum_{\sigma/N_{G}(Q)}(-1)^{|\sigma|}l(kN_{G}(Q)_{\sigma}\Br_{Q_n}(\widehat{e_{Q}})).
\]
Since $O_{p}(N_{G}(Q))\geq Q>1$, it follows that $\sum_{\sigma/N_{G}(Q)}(-1)^{|\sigma|}l(kN_{G}(Q)_{\sigma}\Br_{Q_n}(\widehat{e_{Q}}))=0$ by \cite[Proposition 3.7]{D}.
Therefore, $\sum_{\sigma/\F}(-1)^{|\sigma|}l(kN_{G}(\sigma)e_{n})=0$, which contradicts with the assumption that $\F$ is a  counterexample of (2) of Theorem \ref{equi}.
Thus $O_{p}(\F)=1$.
\end{proof}

\begin{remark}
Let $\F=\F_{(P, e_P )}(kG, b, G)$ be a fusion system of block $b$ of $kG$ with defect group $P$. Assume that $O_{p}(\F)\in \F^{cr}$.
Then 
\[
 \sum_{\sigma\in\F^{cr}/\F}(-1)^{|\sigma|}l(k_{\alpha}\Out_{N_{\F}(\sigma)}(Q_{n}))=
 l(kGb)-l(k_{\alpha}\Out_{\F}(O_{p}(\F))).
\]
\end{remark}

\begin{proof}
Let $Q=O_{p}(\F)$, and let $\sigma\colon 1<Q_{1}<Q_{2}<\cdots<Q_{n}$ be a centric-radical chain. Then we have $Q_{1}\in\F^{cr}$ and $Q\leq Q_{1}$.
Assume that  $\sigma'\colon1<Q_{2}<\cdots<Q_{n}$ if $Q=Q_{1}$ and
$\sigma'\colon1<Q<Q_{1}<Q_{2}<\cdots<Q_{n}$ if $Q<Q_{1}$.
Note that $N_{\F}(\sigma)=N_{\F}(\sigma')$ for all chain of $n>1$ or $n=1$ and $Q<Q_{1}$.
We can remove all the orbits of such pairs from $\sigma\in\F^{cr}/\F$, 
so the only remaining orbits are $1$ and $1<Q$.
Thus the result follows.
\end{proof}

\section{fusion systems with small rank}
In this section, 
we give necessary and sufficient conditions for a fusion system to have (weakly) local rank one or two.
Firstly we investigate the fusion systems of weakly local rank $\mathrm{rank}_{w}\leq 2$.
Let $\F$ be a saturated fusion system over a finite $p$-group $S$.
By definition, $\mathrm{rank}_{w}(\F)=0$ if and only if  $O_p(\F)=S$.

\begin{proposition}\label{rank1}
Let $\F$ be a saturated fusion system over a finite $p$-group $S$.
Then  $\mathrm{rank}_{w}(\F)=1$ if and only if  $\F^e=\{O_p(\F)\}$.
\end{proposition}

\begin{proof}
If $\mathrm{rank}_{w}(\F)=1$, then $O_p(\F)<S$.
Hence the set $\F^e$ of essential subgroups of $\F$ is not empty, 
otherwise, $O_p(\F)=S$ by Alperin's fusion theorem.
Let $Q\in \F^e$.
Then $O_p(\F)\leq Q<S$.
We claim that an $\F$-essential subgroup $Q$ is in $\F^{f\mathfrak{cr}}$.
In fact, since an $\F$-essential subgroup $Q$ is fully $\F$-normalised, $\F$-centric and $\F$-radical, 
it follows that $O_p(N_\F(Q))=Q$ by Lemma \ref{equivalent}.
Therefore, $Q\in\F^{f\mathfrak{cr}}$.
Suppose that $O_p(\F)<Q$.
Then $O_p(\F)< Q<N_S(Q)$ is a normalised weakly centric-radical chain in $\F$.
This is a contradiction.
Thus $Q=O_p(\F)$ and $\F^e=\{O_p(\F)\}$.

Now we assume that $\F^e=\{O_p(\F)\}$.
We claim that $O_p(\F)<S$ is the longest normalised weakly centric-radical chain in $\F$.
So that $\mathrm{rank}_{w}(\F)=1$.
In fact, $Q$ is not in $\F^{\mathfrak{cr}}$ whenever $O_p(\F)< Q<S$.
Since $\F^e=\{O_p(\F)\}$, by Alperin's fusion theorem, we have  $\F=\langle \Aut_\F(O_p(\F)), \Aut_\F(S)\rangle$.
Hence $N_\F(Q)\leq N_\F(S)$ if $O_p(\F)< Q<S$.
Thus $O_p(N_\F(Q))=N_S(Q)>Q$ and $Q\notin\F^{\mathfrak{cr}}$ as claimed.
\end{proof}

\begin{proposition}\label{rank2}
Let $\F$ be a saturated fusion system over a finite $p$-group $S$.
Then  $\mathrm{rank}_{w}(\F)=2$ if and only if
$\F^{e}\setminus\{O_p(\F)\}\neq \emptyset$ and
$N_{\F}(Q)^{\mathfrak{cr}}\setminus\{Q, N_{S}(Q)\}=\emptyset$ whenever $Q\in\F^{\mathfrak{cr}}\setminus\{S, O_p(\F)\}$.
\end{proposition}

\begin{proof}
Suppose that $\mathrm{rank}_{w}(\F)=2$.
Then $\F^{e}\setminus\{O_p(\F)\}\neq \emptyset$ by Proposition \ref{rank1}.
Assume $Q\in\F^{\mathfrak{cr}}\setminus\{S, O_p(\F)\}$.
Then $O_{p}(\F)<Q<N_{S}(Q)$ is one of the longest weakly centric-radical chain in $\F$ as $\mathrm{rank}_{w}(\F)=2$.
So $O_p(N_{N_{\F}(Q)}(P))>P$ for any $Q<P<N_{S}(Q)$.
Hence $N_{\F}(Q)^{\mathfrak{cr}}\setminus\{Q, N_{S}(Q)\}=\emptyset$.

Assume that $\F^{e}\setminus\{O_p(\F)\}\neq \emptyset$.
Let $P\in\F^{e}\setminus\{O_p(\F)\}$. 
Then $O_{p}(\F)<P<N_{S}(P)$ is a normalised centric-radical chain in $\F$.
Thus $\mathrm{rank}_{w}(\F)\geq 2$.
For each $Q\in\F^{\mathfrak{cr}}\setminus\{S, O_p(\F)\}$, 
By hypothesis, we have $N_{\F}(Q)^{\mathfrak{cr}}\setminus\{Q, N_{S}(Q)\}=\emptyset$.
Thus $O_{p}(\F)<Q<N_{S}(Q)$ is one of the longest weakly centric-radical chain in $\F$.
Hence $\mathrm{rank}_{w}(\F)=2$.
\end{proof}

Now we consider the fusion systems $\F$ of local rank $\rank(\F)\leq 2$.

\begin{proposition}\label{rr1}
Let $\F$ be a saturated fusion system over a finite $p$-group $S$.
Then  $\rank(\F)=1$ if and only if  $\F^e=\{O_p(\F)\}$.
\end{proposition}

\begin{proof}
Since $\rank(\F)\neq 0$, we have $O_p(\F)<S$.
Then $\F^{fcr}\neq \{S\}$.
Let $Q\in\F^{fcr}\setminus\{S\}$.
Then $O_p(\F)\leq Q<S$.
If $O_p(\F)<Q$, 
then $O_p(\F)< Q<N_S(Q)$ is a normalised centric-radical chain in $\F$.
This contradicts with  $\rank(\F)=1$.
Hence $O_p(\F)=Q$ and $\F^{fcr}=\{S, O_p(\F)\}$.
As $O_p(\F)<S$, we have $\F^e=\{O_p(\F)\}$.

If $\F^e=\{O_p(\F)\}$, then $\mathrm{rank}_{w}(\F)=1$ by Proposition \ref{rank1}.
Since $\rank(\F)\leq \mathrm{rank}_{w}(\F)$, 
it follows that $\rank(\F)\leq1$.
Thus $\rank(\F)=1$ as $O_p(\F)<S$.
\end{proof}

\begin{proposition}\label{rr2}
Let $\F$ be a saturated fusion system over a finite $p$-group $S$.
Then  $\rank(\F)=2$ if and only if
$\F^{e}\setminus\{O_p(\F)\}\neq \emptyset$ and
$N_\F(Q)^e=\{Q\}$ whenever $Q\in\F^{fcr}\setminus\{S, O_p(\F)\}$.
\end{proposition}

\begin{proof}
Assume that $\rank(\F)=2$.
Then we have $\F^{e}\setminus\{O_p(\F)\}\neq \emptyset$ by Proposition \ref{rr1}.
Given $Q\in\F^{fcr}\setminus\{S, O_p(\F)\}$, since $\rank(\F)=2$, we see that $\rank(N_\F(Q))\leq 1$.
Also $O_{p}(N_\F(Q))=Q<N_{S}(Q)$ as $Q\in\F^{fcr}\setminus\{S, O_p(\F)\}$.
Thus $\rank(N_\F(Q))=1$.
Hence $N_\F(Q)^e=\{Q\}$ by Proposition \ref{rr1}.

Assume that $\F^{e}\setminus\{O_p(\F)\}\neq \emptyset$.
Let $P\in\F^{e}\setminus\{O_p(\F)\}$. 
Then $O_{p}(\F)<P<N_{S}(P)$ is a normalised centric-radical chain in $\F$.
Thus $\mathrm{rank}_{l}(\F)\geq 2$.
Suppose that $N_\F(Q)^e=\{Q\}$ for each $Q\in\F^{fcr}\setminus\{S, O_p(\F)\}$.
Let $\sigma\colon Q_0<Q_1< \cdots<Q_n$ be one of the longest  normalised centric-radical chain in $\F$.
Then $Q_1\in\F^{fcr}\setminus\{S, O_p(\F)\}$ and $N_\F(Q_1)^e=\{Q_1\}$.
By Proposition \ref{rr1}, 
$\rank(N_\F(Q_1))=1$.
As $Q_1< \cdots<Q_n$ is one of the longest  normalised centric-radical chain in $N_\F(Q_1)$, 
hence $n=2$ and $\rank(\F)=2$.
\end{proof}

\noindent\textbf{Acknowledgement.} While writing this paper, J.L. and B.W. were visiting Mathematics Department, University of Aberdeen. They would like to thank both China Scholarship Council and the Department for their support.

\bibliographystyle{amsalpha}

\end{document}